\documentclass{amsart}
\usepackage{amsmath, amssymb, amsthm, mathtools, color}
\usepackage[bookmarks=true]{hyperref}
\usepackage{verbatim}

\numberwithin{equation}{section}

\DeclarePairedDelimiter\abs{\lvert}{\rvert}

\newtheorem{theorem}{Theorem}[section]
\newtheorem{lemma}[theorem]{Lemma}
\newtheorem{corollary}[theorem]{Corollary}

\newtheorem{remark}[theorem]{Remark}
\newtheorem{example}[theorem]{Example}
\newtheorem*{introthA}{Theorem \ref{th:index}}
\newtheorem*{introthB}{Theorem \ref{th:Pointwise}}
\newtheorem*{introthC}{Corollary \ref{co:gisop}}

\newcommand{\Span}{\text{Span}}

\newcommand{\Dim}{\text{Dim}}

\newcommand {\la}{\langle}
\newcommand{\ra}{\rangle}
\newcommand{\Div}{{\text{div}}}
\newcommand{\Hess}{\text{Hess}}
\newcommand{\intS} {\int \limits_\Sigma}

\newcommand{\du} {\, d \mu}
\newcommand{\dA}{\, d \mathcal{A}}
\newcommand{\Au}{\mathcal{A}_\mu}
\newcommand{\Vu}{\mathcal{V}_\mu}
\newcommand{\dV}{\, d\mathcal{V}}
\newcommand{\dAu}{\, d\mathcal{A}_\mu}
\newcommand{\dVu}{\, d\mathcal{V}_\mu}
\newcommand{\Af}{\mathcal{A}_f}
\newcommand{\Vf}{\mathcal{V}_f}
\newcommand{\dAf}{\, d\mathcal{A}_f}
\newcommand{\dVf}{\, d\mathcal{V}_f}

\DeclareMathOperator{\Ind}{Index_{1^\perp}}

\begin{document}
\title[Stable Solutions to the Gaussian Isoperimetric Problem]{The Hyperplane is the Only Stable, Smooth Solution to the Isoperimetric Problem in Gaussian Space}
\date{11/27/2014}
\author{ Matthew McGonagle}
\address{University of Washington, Department of Mathematics, Box 354350, Seattle, WA 98195-4350, USA}
\email{mmcgona1@math.washington.edu}

\author{John Ross}
\address{Department of Mathematics, Johns Hopkins University, 3400 North Charles Street, Baltimore, MD 21218-2686, USA}
\email{jross@math.jhu.edu}

\begin{abstract}
We study stable, two-sided, smooth, properly immersed solutions to the Gaussian Isoperimetric Problem. That is, we study hyper-surfaces $\Sigma^n \subset \mathbb R^{n+1}$ that are second order stable critical points of minimizing $\Au(\Sigma) = \int_\Sigma e^{-|x|^2/4} \dA$ for compact variations that preserve weighted volume. Such variations are represented by $u \in C^\infty_0(\Sigma)$ such that $\int_\Sigma e^{-|x|^2/4} u \dA = 0$. We show that such $\Sigma$ satisfy the curvature condition $H = \la x, N \ra/2 + C$ where $C$ is a constant. We also derive the Jacobi operator $L$ for the second variation of such $\Sigma$.

Our first main result is that for non-planar $\Sigma$, bounds on the index of $L$, acting on volume preserving variations, gives us that $\Sigma$ splits off a linear space. A corollary of this result is that hyper-planes are the only stable  smooth, complete, properly immersed solutions to the Gaussian Isoperimetric Problem, and that there are no hypersurfaces of index one. Finally, we show that for the case of $\Sigma^2 \subset \mathbb R^3$, there is a gradient decay estimate for fixed bound $|C| \leq M$ ($C$ is from the curvature condition) and $\Sigma$ obeying an appropriate $\Au$ condition. This shows that for fixed $C$, in the limit as $R \to \infty$, stable $(\Sigma, \partial\Sigma) \subset (B_{2R}(0), \partial B_{2R}(0))$ with good volume growth bounds approach hyper-planes.
\end{abstract}
\maketitle

\setcounter{section}{0}
%
%
\section*{Introduction}

We will have the need to discuss both volume and area. So, let $\mathcal{A}$ be the $n$-dimensional Hausdorff measure on $\mathbb R^{n+1}$, and let $\mathcal{V}$ be the $(n+1)$-dimensional Hausdorff measure also defined on $\mathbb R^{n+1}$. We consider a Gaussian weighted area measure and a Gaussian weighted volume measure both defined on $\mathbb R^{n+1}$ by $\dAu = e^{-|x|^2/4} \dA$ and $\dVu = e^{-|x|^2/4} \dV$.

In this paper, we always consider two-sided, smooth, properly immersed hyper-surfaces $\Sigma\subset\mathbb R^{n+1}$. $\Sigma$ does not need to be connected, but an orientation must be chosen on each of its components. These orientations provide a globally defined normal vector $N(x)$. We also consider two cases: the case that $\Sigma$ is complete and the case that $\Sigma \subset B_{2R}(0)$ and $\partial\Sigma\subset\partial B_{2R}(0)$. Here, $B_{2R}(0) = \{x\in\mathbb R^{n+1} : |x|\leq 2R\}$ is the euclidean ball of radius $2R$ centered at $x=0$. In the case that $\Sigma$ is complete, we assume that $\Au(\Sigma) < \infty$.

We are interested in hyper-surfaces $\Sigma$ which are stable critical points for the variational problem of minizing $\Au(\Sigma)$ for variations that ``preserve $\Vu$." In Section \ref{sc:firstvar} we will discuss exactly what it means for a variation $F(t,x) : (-\epsilon, \epsilon)\times\Sigma \to \mathbb R^{n+1}$ of an immersed hyper-surface $\Sigma$ to ``preserve $\Vu$". It turns out that these variations are represented by functions $u: \Sigma \to \mathbb R$ defined by $u(x) = \la \partial_t F(0, x), N(x)\ra$ such that $\int_\Sigma u \dAu = 0.$

So, we are specifically interested in hyper-surfaces $\Sigma$ such that the first variation $\delta \Au(u) = 0$ for all $\{u\in C^\infty_0(\Sigma): \int_\Sigma u \dAu = 0 \}$ and such that the second variation $\delta^2 \Au(u) \geq 0$ for all $\{u \in C^\infty_0(\Sigma): \int_\Sigma u \dAu = 0\}$. We will show in Section \ref{sc:firstvar} that the former condition on the first variation is equivalent to the mean curvature condition
\begin{equation}\label{eq:curvcond}
H = \frac {\la x, N(x)\ra}{2} + C,
\end{equation}
where $C$ is a constant. The latter condition on the second variation is equivalent to the stability condition
\begin{equation}\label{eq:minimizing}
Q(u) = -\intS u L u \dAu \geq 0 \text{  for all } u\in C^\infty_0(\Sigma)\text{ such that } \intS u \dAu = 0.
\end{equation}
Here, $Lu$ is defined by
\begin{equation}
Lu = \triangle u - (1/2) \nabla_{x^T} u + |A|^2 u + (1/2)u,
\end{equation}
where $x^T$ is the tangential part of the position vector $x$, $|A|$ is the norm of the second fundamental form of $\Sigma$, and the differential operators $\triangle$ and $\nabla$ are defined using the metric on $\Sigma$ induced by the euclidean metric. Also, $Q(u,v)$ is the quadratic form defined on $C^\infty_0(\Sigma)$ by \begin{equation}
Q(u,v) = - \intS u L v \dAu.
\end{equation}

In this paper, we will use $C$ only to denote the constant in the mean curvature condition \eqref{eq:curvcond}.
One should note that in the case that $C=0$, one has that the $\Sigma$ is a self-shrinker of the mean curvature flow.

In the case that $\Sigma$ is embedded with $\Sigma = \partial \Omega$, we have that the mean curvature condition \eqref{eq:curvcond} and the stability condition \eqref{eq:minimizing} are necessary conditions for $\Sigma$ to be a local minimizer of the Gaussian Isoperimetric Problem. For any region $\Omega$, let $\Omega_r = \{x\in\mathbb R^{n+1}: \text{dist}(x, \Omega) \leq r\}$, where the distance is the Euclidean distance. For the Gaussian Isoperimetric Problem, one is interested in the boundary measure $P_\mu(\Omega)$ defined by
\begin{equation}
P_\mu(\Omega) = \liminf\limits_{r\searrow 0} \frac{\Vu(\Omega_r) - \Vu(\Omega)}{r}.
\end{equation} 
The Gaussian Isoperimetric Problem is concerned with minimizing $P_\mu(\Omega)$ over regions $\Omega$ with fixed Gaussian volume $\Vu(\Omega) = \mathcal{V}_0$. For $\partial\Omega$ smooth, one has that $P_\mu(\Omega) = \Au(\Omega)$. For details, see the survey by Ros \cite{ros}. So we indeed see that conditions \eqref{eq:curvcond} and \eqref{eq:minimizing} are necessary for $\Omega$ with $\Sigma = \partial \Omega$ to be a local minimizer to the Gaussian Isoperimetric Problem. 

It is important to note that the Gaussian Isoperimetric Problem is not the same as the classical Isoperimetric Problem for a metric conformal to the standard metric on $\mathbb R^{n+1}$; this is readily seen by the fact that $\dAu$ and $\dVu$ are defined using the same weights. Therefore, the mean curvature condition \eqref{eq:curvcond} is not the same as $\Sigma$ having constant mean curvature in a metric conformal to $\mathbb R^{n+1}$.

Borell \cite{borell} and Sudakov \& Tsirel'son \cite{sudakovtsirelson} show that the global minimizers $\Omega$ for the Gaussian Isoperimetric Problem are half-spaces. Specifically, they show that for any Borel $\Omega$ such that $\Vu(\Omega) = \mathcal{V}_0$ and any half-space $S$ such that $\Vu(S) =\mathcal{V}_0$, one has that $\Vu(\Omega_r) \geq \Vu(S_r)$.

The main result of our paper includes a similar fact for critical hyper-surfaces satisfying the curvature condition \eqref{eq:curvcond} and the locally minizing condition \eqref{eq:minimizing}. Before we state it, let us define $\Ind Q$ to be the maximal dimension of subspace $W\subset\{u\in C^\infty_0(\Sigma): \int_\Sigma u\dAu =0\}$ for which $Q$ is negative definite on $W$. The notation being suggestive of the fact that $\Ind Q$ is the index of the quadratic form $Q$ over the subspace of $C^\infty_0(\Sigma)$ orthogonal to the constant functions. Then, we have the following:

\begin{introthC}
The hyper-planes are the only two-sided, smooth, complete, properly immersed hypersurfaces $\Sigma\subset\mathbb R^{n+1}$ such that $\Au(\Sigma)<\infty$, $\Sigma$ satisfies $H = (1/2)\la x, N\ra + C$ for some constant $C$, and $\Sigma$ satisfies the locally stable condition \eqref{eq:minimizing}. 

Furthermore, there are no two-sided, smooth, complete, properly immersed $\Sigma$ such that $\Au(\Sigma)<\infty$, $\Sigma$ satisfies $H = (1/2)\la x, N \ra + C$, and $\Ind Q = 1$.  
\end{introthC}

This corollary stands in contrast with the minimal case. The self-shrinkers of mean curvature flow ($C = 0$) are exactly those hyper-surfaces that satisfy $\delta \Au(u) = 0$ for all $u\in C^\infty_0(\Sigma)$. Colding-Minicozzi \cite{cm2012} show that there are no self-shrinkers that are stable for all variations $u\in C^\infty_0(\Sigma)$. This non-existence of stability has been generalized to other measures by Impera-Rimoldi \cite{ImperaRimoldi} and Cheng-Mejia-Zhou \cite{ChengMejiaZhou}.

This corollary is the direct result of a more general statement related to the fact that bounds on $\Ind Q$ force any non-planar $\Sigma$ to split off a linear space. An explicit statement is contained in the following.
\begin{introthA}
Consider any two-sided, smooth, properly immersed, non-planar hyper-surface $\Sigma\subset\mathbb R^{n+1}$ such that $\Au(\Sigma)<\infty$, $\Sigma$ satisfies the mean curvature condition $H = (1/2)\la x, N \ra + C$, and $\Ind Q \leq n$. Then there exists an $i$ such that $ n+1-\Ind Q \leq i \leq n$, and we have that
\begin{equation}
\Sigma = \Sigma_0 \times \mathbb R^{i}.
\end{equation}
Furthermore, for such non-planar $\Sigma$ it is impossible that $\Ind Q = 0$ or \\$\Ind Q=1$.
\end{introthA}

The main idea of the proof of this theorem uses translations by constant vectors $v\in\mathbb R^{n+1}$ represented by functions $u = \la v, N \ra$, and it also uses uniform movement in the normal direction which is represented by the constant function $u=1$. For any $\Sigma$ satisfying the curvature condition \eqref{eq:curvcond} and any constant vector $v \in \mathbb R^{n+1}$, we have that $L \la v, N\ra = (1/2) \la v, N \ra$. Then, using that $W = \{u = \la v, N \ra : v\in\mathbb R^{n+1}\}$ is a vector space, we can find $v$ such that $u= \la v,N\ra$ satisfies $\int_\Sigma u \dAu  =0.$ In fact, by being careful with our estimates we can consider the space $\Span \{ \la v,N\ra, 1\}_{v\in\mathbb R^{n+1}}$. By applying appropriate cut-off functions and plugging into the stability condition \eqref{eq:minimizing}, we get a lower bound on the dimension of $\{v\in\mathbb R^{n+1}:\la v, N \ra \equiv 0  \}$.

It is interesting to compare our use of translations with the use of homotheties in the classical Isoperimetric Problem. It is a well known result of Schwarz \cite{schwarz} and Steiner \cite{steiner} that the global minimizers of the classical Isoperimetric Problem for $\mathbb R^{n+1}$ are round $n$-spheres. Barbosa-Do Carmo \cite{barbosadocarmo} use functions representing homotheties and uniform movement in the normal direction to show that the round $n$-spheres are also the only compact smooth hyper-surfaces which are local minimizers for the Isoperimetric Problem in $\mathbb R^{n+1}$. Homotheties also play an important role in other work such as that of Morgan-Ritor{\'e} \cite{morganritore}, Palmer \cite{palmer}, and Wente \cite{wente}.

For the case of $\Sigma \subset B_{2R}(0)$ and $\partial \Sigma \subset \partial B_{2R}(0)$, we are able to show a type of estimate for $\int_{\Sigma\cap B_R(0)} |A|^2 \dAu$. For the case of $n=2$, we use this integral estimate to obtain pointwise decay estimates for $|A|$ depending only on $R,$ $M$ such that $|C| \leq M$, and appropriate $\Au$ conditions. Specifically, we have
\begin{introthB}[Pointwise for $n=2$]
Let $M > 0$ be given and $R> 1$. Also, let $\Sigma \subset B_{2R}(0) \subset \mathbb R^{3}$ with $\partial \Sigma \subset \partial B_{2R}(0)$ be a hyper-surface with $H = \la x, N \ra/2 + C$ and $|C| \leq M$ that satisfies the stability condition \eqref{eq:minimizing}.  

There exists $\epsilon_M > 0$ such that if $\Au(\Sigma \cap B_R) \geq 2B_n R^{-2}\Au(\Sigma\cap(B_{2R}\setminus B_R))$ and $\Au(\Sigma \cap(B_{2R} \setminus B_R)) < (R^2/2B_n)e^{-(\frac{1}{16}+\gamma) R^2} \epsilon_M $ for some $\gamma > 0$, then

\begin{equation}
\sup\limits_{x\in B_{R/4}(0)} |A|^2 \leq 16 R^2e^{-\gamma R^2}.
\end{equation}
\end{introthB}
Note that if $\Sigma$ is complete and $\Au(\Sigma)<\infty$, then for large $R$ we may apply this theorem and then take $R\to\infty$. Doing so, we again see that the only complete, two-sided, smooth, properly immersed $\Sigma$ that satisfy the mean curvature condition \eqref{eq:curvcond} and the stability condition \eqref{eq:minimizing} are the hyper-planes.

\subsection{Conventions and Notation}

We use $\Sigma$ to represent hyper-surfaces that are smooth, two-sided, and properly immersed. We will use $x$ to denote the position vector in $\mathbb R^{n+1}$.

We will denote the second fundamental form of $\Sigma$ by $A(X,Y) = \la \nabla_X N, Y \ra$. We then use $H = \Div N$ for the mean curvature of $\Sigma$.  With this convention, the cylinder $S^{k} \times \mathbb{R}^{n-k} \subset \mathbb{R}^{n+1}$ with radius $R$ has mean curvature $H = k/R$.

$C$ will always be used to denote the constant in the mean curvature condition \eqref{eq:curvcond}. When we need to make use of other constants, we will use different letters.

For volume forms and area forms we will use the following notation:
\begin{itemize}
\item $\mathcal{V}$ is the $(n+1)$-Hausdorff measure on $\mathbb R^{n+1}$.
\item $\mathcal{A}$ is the $n$-Hausdorff measure on $\Sigma\subset\mathbb R^{n+1}$.
\item $\dVu = e^{-|x|^2/4} \dV$.
\item $\dAu = e^{-|x|^2/4} \dA$.
\item $\dVf = e^f \dV$ for a function $f: \mathbb R^{n+1} \to \mathbb R$.
\item $\dAf = e^f \dA$ for a function $f: \mathbb R^{n+1} \to \mathbb R$.
\end{itemize}

We will use the second order operators defined by
\begin{itemize}
\item $L u = \triangle u - (1/2)\nabla_{x^T} u + |A|^2 u + (1/2) u.$
\item $\mathcal L u = \triangle u - (1/2)\nabla_{x^T} u.$
\end{itemize}
Note that $Lu = \mathcal L u + |A|^2 + (1/2) u$. Here $x^T$ is the tangential part of the position vector.

Let $Q(u,v) = -\int_\Sigma u L v \dAu$ for $u,v \in C^\infty_0(\Sigma)$.

$\Ind Q$ denotes the maximal dimension of subspace of $W\subset C^\infty_0(\Sigma)$ such that $Q$ is negative definite on $W$.

We let $A_{ij,k}$ denote $\nabla_{e_k} A_{ij}$.  The Codazzi equation then says that $(A_{ij,k})$ is fully symmetric.  The Riemann curvature tensor will be denoted by $R_{ijkl} = \la \nabla_i\nabla_j \partial_k - \nabla_j \nabla_i \partial_k, \partial_l \ra$.  Since our ambient manifold is $\mathbb{R}^{n+1}$, the Gauss equation is given by
\begin{equation}
\label{eq:03}
R_{ijkl}=A_{il}A_{jk} - A_{ik}A_{jl}
\end{equation}

$\omega_n$ will denote the $n$-volume of the unit ball in $\mathbb R^n$.

\subsection{Structure of This Paper}
The paper will be organized as follows:

 In Sections 1-2, we develop the problem we are trying to solve. In Section \ref{sc:firstvar}, we discuss the first and second variations of both the weighted area $\Af$ and the weighted volume $\Vf$ for a general weight $e^f$.  We also discuss what it means for a compact normal variation of an immersed surface to be \emph{variation preserving $\Vf$.}

In Section \ref{sc:gauss}, we consider the explicit case of Gaussian weighted area $\dAu$ and Gaussian weighted volume $\dVu$. This includes explicit formulas for variations and the Jacobi operator.

 In Section \ref{sc:planes}, we prove that hyper-planes satisfy both the mean curvature condition \eqref{eq:curvcond} and the stability condition \eqref{eq:minimizing}.

 In Section \ref{sc:index}, we prove Theorem \ref{th:index} on how bounds of $\Ind Q$ force $\Sigma$ to split off a linear space.

 In Section \ref{sc:intest}, we prove an estimate for $\int_{\Sigma\cap B_R(0)} |A|^2 \dAu$ that will be necessary to establish our pointwise curvature decay estimates. In Section \ref{sc:ptest}, we use the integral estimate to obtain pointwise decay estimates for $|A|$.

 Appendix \ref{sc:mean} is devoted to reformulating the Mean Value Inequality for minimal hyper-surfaces to a form useful for hyper-surfaces satisfying the mean curvature condition \eqref{eq:curvcond}. Explicitly, we show a mean value type inequality for hyper-surfaces $\Sigma$ with bounded $|H|\leq M$.


\section*{Acknowledgements}
The authors would like to thank Professor William Minicozzi for his valuable support and direction. They would also like to thank the reviewer for their valuable help and suggestions.

%
%

\section{First and Second Variations} \label{sc:firstvar}

We will discuss the relationship between the mean curvature condition \eqref{eq:curvcond} and $\delta \Au$, and we will also discuss the relationship between the stability condition \eqref{eq:minimizing} and $\delta^2 \Au$. First, in this section we compute the first and second variations for an arbitrary weight $e^f$ where $f: \mathbb R^{n+1}\to\mathbb R$. We define a weighted area measure by $\dAf = e^f \dA$ and a weighted volume measure by $\dVf = e^f \dV.$ We call a compact variation $F(t,x): (-\epsilon, \epsilon)\times\Sigma \to \mathbb R^{n+1}$ a normal variation if $\partial_t F$ is normal to $\Sigma_t = F(t,\Sigma)$ for all $t$.

Let us motivate what it means for a normal variation $F(t,x) : (-\epsilon, \epsilon) \times \Sigma \to \mathbb R^{n+1}$ to ``preserve $\Vf$" in the case that $\Sigma$ is immersed. Consider the case that $\Sigma$ is embedded and $\Sigma = \partial \Omega$. Let $\Sigma_t = F(t, \Sigma)$ and $\Omega_t = F(t, \Omega)$. In the case of a variation preserving $\Vf$, we have that $\Vf(\Omega_t) = \Vf(\Omega)$. One has that $\partial_t\Vf(\Omega_t) = \int_{\Sigma_t} \la \partial_t F(t,x), N(x) \ra \dAf = 0$. So, we generalize the concept of a normal variation $F(t,x)$ preserving $\Vf$ to immersed hyper-surfaces $\Sigma$ by defining that $\int_{\Sigma_t} \la \partial_t F, N \ra \dAf = 0$ for all t.

It is now clear that if a normal variation $F(t,x)$ preserves $\Vf$, then we have that $u(x) = \la \partial_t F(0,x), N(x)\ra$ satisfies $\int_\Sigma u \dAf = 0.$ An argument similar to that given in Barbosa-do Carmo \cite{barbosadocarmo} shows the opposite is true as well:
\begin{lemma} \label{lm:volpresexist}
Let $\Sigma$ be an immersed two-sided hyper-surface, and let $u\in C^1_0(\Sigma)$ such that $\int_\Sigma u \dAf = 0$. Then, there exists a compact normal variation $F(t,x): (-\epsilon,\epsilon)\times\Sigma\to \mathbb R^{n+1}$ such that $u(x) = \la\partial_t F(0,x),N(x)\ra$ and $F(t,x)$ is $\Vf$ preserving.
\end{lemma}
So we see that $\{u\in C^\infty_0(\Sigma) : \int_\Sigma u \dAf = 0\}$ is the correct subspace of functions to restrict to when considering normal variations preserving $\Vf$.

We now compute the first variation of the weighted area functional. It is standard that $\partial_t (\dA) = uH\dA$, and it is clear that $\partial_t e^f = e^f u \la \nabla f, N \ra.$ For any compact normal variation $F(t,x)$, let $u(t,x) = \la \partial_t F(t, x), N(t,x) \ra$. Then we have that
\begin{equation}\label{eq:1_1}
\partial_t \Af(\Sigma_t) = \int\limits_{\Sigma_t} u \left( H + \la \nabla f , N \ra\right) \dAf.
\end{equation}

By using pairs of approximations to the identity with opposite weights and centered at different points for $u$, we find the following curvature condition must be satisfied by critical hyper-surfaces of $A_f$ for all normal variations preserving $V_f$.

\begin{lemma}\label{lm:1_2}
$\Sigma$ is a hypersurface satisfying $\delta A_f(u) = 0$ for all $\{u \in C^\infty_0(\Sigma): \int_\Sigma u \dAf = 0\}$ if and only if $H+\la \nabla f, N \ra$ is constant on $\Sigma$.
\end{lemma}


Now we discuss the second variation $\delta^2\Af(u)$ for $u$ representing normal variations of $\Sigma$ that are $\Vf$ preserving, i.e. $u\in\{v\in C^\infty_0(\Sigma): \int_\Sigma v \dAf = 0\}$.

\begin{lemma}\label{lm:1_3}
Let $\Sigma$ satisfy $\delta A_f(u) = 0$ for all $u\in C^\infty_0(\Sigma)$ such that $u$ represents a normal variation preserving $\Vf$, i.e. $\int_\Sigma u \dAf = 0$. Then for any $u = \partial_t F(0,x)$ representing a normal variation preserving $\Vf$, we have that
\begin{equation}
 \delta^2 \Af(\Sigma)(u) = -\intS u \left(\triangle u + |A|^2 u - \Hess_f(N,N)u + \la \nabla f, \nabla u\ra\right) \dAf. \label{eq:513}
\end{equation}

\end{lemma}

\begin{proof}
Given $u\in\{v\in C^\infty_0(\Sigma): \int_\Sigma v\dAf = 0\}$, let $F(t,x)$ be a normal variation preserving $\Vf$ such that $u(x) = \la \partial_t F(0, x)\ra$. Extend $u(x)$ to be $u(x,t) = \la \partial_t F(t,x), N(x,t)\ra$. Then we have that $\int_{\Sigma_t} u \dAf = 0$ for all $t$.

Let $H + \la \nabla f, N \ra = C$ a constant function on $\Sigma$. From \eqref{eq:1_1}, we have that
\begin{align}
\delta^2 \Af(\Sigma)(u) & = \partial_t\left(\int\limits_{\Sigma_t} (H + \la \nabla f, N \ra )u\dAf \right)\\
& = \intS \partial_t(H + \la \nabla f, N \ra) (u\dAf) + \intS (H+\la \nabla f, N \ra) \partial_t (u\dAf),\\
& = \intS \partial_t(H + \la \nabla f, N \ra) (u\dAf) + C \partial_t\left(\int\limits_{\Sigma_t} u\dAf\right), \\
& =  \intS \partial_t(H + \la \nabla f, N \ra) (u\dAf) .
\end{align}
 It is standard in the literature that for Euclidean space we have that $\partial_t N = -\nabla u$ and that $\partial_t H = -\triangle u - |A|^2 u$. For an explanation of these formulas, see Colding-Minicozzi \cite{cm2012}. Also, it is clear that $\la \partial_t \nabla f, N \ra = \Hess_f(N,N)u^2$. Therefore we get that
 \begin{equation}
 \delta^2 \Af(\Sigma)(u) = -\intS u \left(\triangle u + |A|^2 u - \Hess_f(N,N)u + \la \nabla f, \nabla u\ra\right) \dAf.
 \end{equation}

\end{proof}

%
%

\section{The Gaussian Measure} \label{sc:gauss}

Here we consider the $(n+1)$-dimensional Gaussian space modeled by $\mathbb R^{n+1}$ with the weighted volume measure $\dVu = e^{-|x|^2/4}\dV$. We also consider the associated weighted area form $\dAu = e^{-|x|^2/4}\dA$. Results on the first variation and the second variation are provided by our lemmas in Section \ref{sc:firstvar} when using $f = -|x|^2/4$. One should note that the weighted measures $\dVu$ and $\dAu$ have the same weight $e^{-|x|^2/4}$. It is for this reason that our variational problem is not equivalent to a classical Isoperimetric Problem for a metric that is conformal to the standard metric on $\mathbb R^{n+1}$. This is in contrast to the case of self-shrinkers of the mean curvature flow, which may be realized as minimal sub-manifolds of $\mathbb R^{n+1}$ for a metric $g_{ij}=e^{-|x|^2/2n}\delta_{ij}$, see Colding-Minicozzi \cite{cm2012}.

Applying \eqref{eq:1_1}, we see that
\begin{equation}
\delta \Au(\Sigma)(u) = \intS u\left(H - \frac{1}{2}\la x, N \ra\right) \du\label{eq:62}.
\end{equation}
Now, we note that the compact normal variations preserving $\Vu$ are exactly represented by the space $\{u\in C^\infty_0(\Sigma) : \int_\Sigma u \dAu = 0\}$.

From Lemma \ref{lm:1_2}, we see that the condition $\delta \Au(\Sigma)(u) = 0$ for all $u\in C^\infty_0(\Sigma)$ preserving $\Vf$ is equivalent to
\begin{equation}
H = \frac{1}{2} \la x, N \ra + C,\label{eq:63}
\end{equation}
where $C$ is a constant on $\Sigma$.  

We point out some readily available examples of hyper-surfaces satisfying \eqref{eq:63}:
\begin{example}
Any hyper-plane in $\mathbb R^{n+1}$ (not necessarily passing though the origin) satisfies \eqref{eq:63}, as $\la x, N \ra$ is constant and $H=0$.
\end{example}
\begin{example}
For any sphere $|x - x_0| = R$, we have that $N = \frac{x-x_0}{R}$ and $H = \frac{n}{R}$. We see that $\la x, N \ra = R + (1/R) \la x_0, x-x_0\ra.$ Hence, $\frac{n}{R} - (R/2) - (1/2R)\la x_0, x-x_0\ra = C$ a constant. Therefore, we see that the only spheres that satisfy \eqref{eq:63} are the spheres centered at the origin.
\end{example}
\begin{example}
For any cylinder $\{ (x,y) \in S^k \times \mathbb R^{n-k} : |x-x_0|=R\}$, we have that $N = (1/R)(x-x_0)$ and $H = k/R$. Hence, $k/R - R/2 - (1/2R)\la x_0, x-x_0\ra = C$ a constant. Therefore, the only cylinders that satisfy \eqref{eq:63} are those that are cylinders over spheres $S^k$ in some $(k+1)$-plane and centered at the origin.
\end{example}
\begin{example}
As noted before, if $C = 0$, then the hyper-surfaces satisfying \eqref{eq:63} are the self-shrinkers  of the mean curvature flow. There are many examples of these self-shrinkers, including Angenent's self-shrinking torus \cite{angenent} and the noncompact examples of Kapouleas, Kleene, and M{\o}ller \cite{KKM2011} and Nguyen \cite{nguyen}.
\end{example}

Using Lemma \ref{lm:1_3}, we see that for any hyper-surface $\Sigma$ satisfying \eqref{eq:63} and any compact normal variation preserving $\Vu$ represented by $u\in\{v\in C^\infty_0(\Sigma):\int_\Sigma v \dAu = 0\}$, we have that
\begin{equation}\label{eq:ddAg}
\delta^2 \Au(\Sigma)(u) = -\int\limits_\Sigma u L u \dAu,
\end{equation}
where $L$ is defined by
\begin{equation}
L u = \triangle u - \frac{1}{2}\nabla_{x^T}u + |A|^2 u + \frac{1}{2}u.
\end{equation}
Here, $x^T$ is the tangential part of the position vector, and both of the operators $\triangle$ and $\nabla$ are defined using the metric on $\Sigma$ induced by the Euclidean metric. We will also make use of the second and first order parts of $L$ which we will call $\mathcal L$. So $\mathcal L$ is defined by
\begin{equation}
\mathcal{L} u = \triangle u - \frac{1}{2}\nabla_{x^T} u.
\end{equation}
Both $L$ and $\mathcal L$ are self-adjoint with respect to $\dAu$. Furthermore, we have that $\int_\Sigma \la \nabla f, \nabla g\ra \dAu = - \int_\Sigma f \mathcal L g \dAu.$

For convenience, we also define the quadratic form $Q$ on $C^\infty_0(\Sigma)$ by $Q(u,v) = -\int_\Sigma u L v \dAu$.

We then say that $\Sigma$ is \emph{$\Vu$ preserving stable} (or sometimes just \emph{stable}) if $Q(u,u) \geq 0$ for every $\Vu$ preserving compact normal variation represented by the function $u$ (i.e. stability condition \eqref{eq:minimizing}). Remember, that Lemma \ref{lm:volpresexist} shows us that every such variation is represented by such a $u$, and vice versa.

%
%

\section{Stability of the Hyper-planes} \label{sc:planes}

\begin{theorem}
Hyperplanes are $\Vu$ preserving stable hyper-surfaces.
\end{theorem}

\begin{proof}
We begin by observing that if a hyper-plane does not pass through the origin, then without any loss in generality, it may be considered to be the plane $x_{n+1} = c$.  A change of variables $x \rightarrow x - (0,0,...,c)$ shifts this plane to pass through the origin and changes the quadratic functional $Q$ by the constant factor $e^{-|c|^2 / 4}$. Therefore, it suffices to consider the stability of a hyper-plane through the origin.

For such a hyper-plane, since $A\equiv 0$, the operator $L$ takes the form 
\begin{equation}
L u = \left(\triangle - \frac{1}{2}\nabla_{x^T}+\frac{1}{2}\right)u.
\end{equation}
Such an operator is known to be comparable to the harmonic oscillator, see Kapouleas-Kleene-M\o ller \cite{KKM2011}. We temporarily remove our restriction to volume preserving variations and consider all possible variations.  The operator $L$ may be factored as
\begin{equation}
L u = e^{|x|^2 / 8} \left(\triangle - \frac{|x|^2}{16} + \frac{n+2}{4}\right) e^{-|x|^2 / 8}u\label{eq:71}.
\end{equation}
Here, the new operator
\begin{equation}
H_x \equiv \left(\triangle - \frac{|x|^2} { 16} + \frac{n+2}{4}\right) \label{eq:72}
\end{equation}
is being applied to $e^{-|x|^2 / 8}u$.  $H_x$ may be viewed as a shifted version of the harmonic oscillator
\begin{equation}
\tilde{H} \equiv \triangle - |x|^2 \label{eq:73}
\end{equation}
Indeed, a change of variables $x = 2y$ gives us the operator
\begin{equation}
H_y \equiv \frac{1}{4}\triangle - \frac{|x|^2} { 4} + \frac{n+2}{4}, \label{eq:74}
\end{equation}
and $\tilde{H} = 4H_y - (n+2)$.

The eigenvalues of $\tilde{H}$ (i.e. $\lambda$ such that $\tilde{H} u = -\lambda u$) are well-known to be $n + 2k$ for $k = 0,1,2,$..., and the eigenvectors are products of $e^{-|y|^2 / 2}$ and Hermite polynomials.  So the eigenvalues of $H_y$ (which are equivalent to the eigenvalues of $L$) take the form $(k-1)/2$.  Except for the first eigenvalue, these are all positive.

Observe that $n$ is the lowest eigenvalue of $\tilde{H}$ and has an eigenspace spanned by $e^{-|y|^2 / 2}$.  Undoing the change of variables, the lowest eigenvalue of $L$ is $-1/2$.  Furthermore, the lowest eigenspace of $L$ is spanned by the constant functions. Note that if $u\in C^\infty_0(\Sigma)$ such that $\int_\Sigma u \dAu = 0$, then $u$ is orthogonal to the constant functions under the weighted $\dAu$ measure. Since the other eigenvalues of $L$ are non-negative, we then have that the hyper-planes are $\Vf$ preserving stable hyper-surfaces.

\end{proof}

%
%

\section{Index of Non-planar Hyper-surfaces} \label{sc:index}


We now want to examine the the index of nonplanar hypersurfaces. We'll begin with the compact case (which is very simple because we do not need any cutoff functions), and then generalize to the non-compact case. First, however, we remark on an important identity.

\begin{lemma}
Let $\Sigma \subset \mathbb R^{n+1}$ satisfy $H = (1/2)\la x, N \ra + C$, and let $v \in \mathbb R^{n+1}$ be a constant vector. Then 
\begin{align}
L\la v,N \ra &= \frac{1}{2}\la v,N \ra\label{eq:81}
\end{align}
\end{lemma}

\begin{proof}
The proof is identical to the corresponding proof for self-shrinkers in Colding-Minicozzi \cite{cm2012}. In particular, the main computation on $\nabla H$ is the same for self-shrinkers and for our hypersurfaces.
\end{proof}




\begin{lemma}\label{lm:compact}
Let $\Sigma$ be a compact hypersurface that satisfies $H = (1/2)\la x, N \ra + C$. Then $Q$ is negative definite on $\Span\{1, \la v, N\ra :v\in\mathbb R^{n+1}\}$. 
\end{lemma}

\begin{proof}
From \eqref{eq:81} we have that $Q$ is negative definite on $\Span\{\la v, N \ra : v\in\mathbb R^{n+1}\}$. So it is sufficient to check that $Q(1+u, 1+u)<0$ for some $u = \la v, N\ra$.

We use the divergence theorem and $Lu= (1/2) u$ to get that
\begin{align}
\frac{1}{2}\int_\Sigma u \dAu & = \intS L u \dAu, \\ 
& = \intS \mathcal{L} u \dAu + \intS \left(|A|^2 + \frac{1}{2}\right) u \dAu, \\
& = \intS  \left(|A|^2 + \frac{1}{2}\right) u \dAu.
\end{align}
Therefore, $\int_\Sigma |A|^2 u \dAu = 0$, and so $u$ is orthogonal to $|A|^2$. We then compute 
\begin{equation}
Q(1+u, 1+u) =  -\int_\Sigma \left((1/2)+|A|^2\right)\dAu - \int_\Sigma u^2 \dAu,
\end{equation}
where we have used the self-adjointness of $L$ and the fact that $u$ is orthogonal to $|A|^2$. This shows that $Q$ is negative definite on $\Span\{1, \la v, N\ra :v\in\mathbb R^{n+1}\}$.

\end{proof}

We now turn to the noncompact case. The argument is morally similar to the compact case, but we need cutoff functions to make the integrals work. For the proof of Lemma \ref{lm:compact}, the orthogonality of $|A|^2$ and $\la v, N \ra$ was critical. For the non-compact case, we don't have that $|A|^2$ and $\phi^2 \la v, N \ra$ are orthogonal. However, \eqref{eq:orthoA} gives us control on their product.

\begin{lemma}\label{lm:4_1}
For any functions $\phi \in C^\infty_0(\Sigma)$ and $f\in C^\infty(\Sigma)$ we have that\
\begin{equation}
\intS \phi f L(\phi f) \dAu = \intS \phi^2 f Lf \dAu - \intS |\nabla \phi|^2 f^2 \dAu. \label{eq:Lphif}
\end{equation}
Also, for any $\Sigma$ satisfying $H = \la x, N\ra / 2 + C$  and constant vector $v\in \mathbb R^{n+1}$ we have that
\begin{equation}
\intS \phi^2|A|^2\la v, N \ra \dAu = 2 \intS \phi A(\nabla \phi, v^T) \dAu \label{eq:orthoA},
\end{equation}
\end{lemma}

\begin{proof}
\begin{align}
\intS \phi f L(\phi f) \dAu & = \intS \left(f^2\phi \mathcal L \phi + \frac{1}{2}\la \nabla \phi^2, \nabla f^2 \ra+ \phi^2 f L f\right)  \dAu \notag,\\
& = \intS (f^2\phi \mathcal L \phi  - f^2 \phi \mathcal L \phi - f^2 |\nabla \phi|^2 + \phi^2 fLf) \dAu \notag,\\
& = \intS(\phi^2 fLf - |\nabla \phi|^2 f^2)\dAu \label{eq:810}.
\end{align}
This shows \eqref{eq:Lphif}.

To prove \eqref{eq:orthoA}, consider a $\Sigma$ satisfying $H = (1/2)\la x, N\ra + C$. Using $L\la v, N \ra = (1/2)\la v, N\ra$ and \eqref{eq:Lphif} we have that
\begin{align}
\frac{1}{2} \intS \phi \la v, N \ra \dAu & = \intS \phi L \la v, N \ra \dAu \notag, \\
& = \intS \left(\frac{1}{2} + |A|^2\right) \phi \la v, N \ra \dAu - \intS \la \nabla \phi, \nabla \la v, N \ra \ra \dAu \label{eq:811}.
\end{align}
Therefore, we get that
\begin{equation}
\intS |A|^2 \phi \la v, N \ra \dAu = \intS A(\nabla \phi, v^T) \dAu \label{eq:812}.
\end{equation}
Substituting $\phi \to \phi^2$, we get \eqref{eq:orthoA}.

\end{proof}



\begin{lemma}\label{lem:phiV}

Let $\Sigma$ be a hyper-surface such that $H = (1/2)\la x, N\ra + C$ and $\Au(\Sigma)<\infty$. Then there exists $\phi \in C_0^\infty(\Sigma)$ such that $Q$ is negative definite on $\phi V$ and $\Dim (\phi V) = \Dim V$, where $$V \equiv \Span\{1, \la v, N\ra : v \in \mathbb{R}^{n+1}\}.$$

\end{lemma}

\begin{remark}
The vector space $V$ consists of all functions that are spanned by the constant function and those functions that correspond to translating the hypersurface. It is not true that all functions in $V$ represent variations preserving $\Vu$.
\end{remark}

\begin{proof}

Consider $u \equiv c_0 + \la v, N \ra$, and consider $Q(u\phi, u\phi)$.
From \eqref{eq:Lphif} we have that
\begin{align}
Q(\phi u, \phi u) & = -\intS \phi^2 u L u \dAu + \intS |\nabla \phi|^2 u^2 \dAu \label{eq:813},\\
& = -\intS \phi^2 u c_0 (1/2 + |A|^2) \dAu - \intS \phi^2 u (1/2) \la v, N \ra \dAu + \intS |\nabla \phi|^2 u^2 \dAu.
\end{align}
Using that $u = c_0 + \la v, N \ra$, we get
\begin{equation}\label{eq:4_13}
Q(\phi u, \phi u)  = -(1/2)\intS \phi^2 u^2 \dAu - \intS \phi^2 |A|^2 c_0^2 \dAu - \intS \phi^2 |A|^2 c_0 \la v, N \ra \dAu + \intS |\nabla\phi|^2 u^2 \dAu.
\end{equation}
Now, using \eqref{eq:orthoA} and a Cauchy-Schwarz inequality of the form $2ab \leq a^2 + b^2$, we get
\begin{align}
\left|\intS \phi^2 |A|^2 c_0 \la v, N \ra\dAu\right| & = 2\left| \intS \phi A(\nabla \phi, v^T)c_0 \dAu \right |,\notag\\
& \leq \intS \phi^2 |A|^2 c_0^2 \dAu + \intS |\nabla \phi|^2|v^T|^2\dAu.\label{eq:814}
\end{align}
Therefore,
\begin{equation}
Q(\phi u, \phi u) \leq -(1/2)\intS\phi^2 u^2 \dAu + \intS |\nabla\phi|^2 (u^2 + |v^T|^2) \dAu \label{eq:815}
\end{equation}

Now fix a point $p\in\Sigma$ and let $r$ be the Euclidean distance from the origin. For $R>0$ large, define the cut-off function $\phi_R$ such that
\begin{equation}
\phi_R(r) =
\begin{cases}
1 & r\leq R \\
1 - (1/R)(r-R) & R \leq r \leq 2R \\
0 & r\geq 2R.
\end{cases}\label{eq:816}
\end{equation}
and observe that $|\nabla \phi_R| \leq 1/R$. Also note that $\phi_R \in C^\infty_0(\Sigma)$ since $\Sigma$ is proper.

Now, note that \eqref{eq:815} becomes
\begin{equation}\label{eq:4_17}
Q(\phi_R u, \phi_R u) \leq -\frac{1}{2} \intS \phi_R^2 u^2 \dAu + \frac{2}{R^2}\int\limits_{\Sigma \setminus B_R(0)} u^2 \dAu.
\end{equation}
Since $\Au(\Sigma)<\infty$, we have that $\int_\Sigma u^2 \dAu < \infty$ for any $u \in V$. For fixed $u \neq 0$ and taking $R\to \infty$, we see that there exists $R_u$ such that $Q(\phi_{R_u} u , \phi_{R_u} u) < 0$. We wish to show that we can find such an $R$ that is independent of $u \in V$.

Now, note that the dimension of $V$ is not necessarily $n+1$. Let $\{ c_i + \la v_i, N \ra \}$ be a basis for $V$ with $|c_i|^2 + |v_i|^2 = 1$. Define $S \equiv \{ d^i (c_i + \la v_i, N \ra) : \sum d_i^2 = 1\}$. Since $\Dim V < \infty$, we have that $S$ is compact. This implies there exists $R_0$ such that for all $u\in S$ we have that $B_{R_0} \cap \{ u \neq 0 \} \neq \emptyset$; if not, there would exist a sequence of $ u_j = d^i_j (c_i + \la v_i, N \ra)\in S$ such that $u_j \equiv 0$ on $B_j$. By passing to a subsequence and taking a limit $d^i_j \to d^i_\infty$, this implies that $d^i_\infty (c_i + \la v_i, N \ra) \equiv 0 \in S$ which is a contradiction.

Therefore, for $R \geq R_0$ and all $u \in S$ we have that
\begin{equation}
\intS \phi_R^2 u^2 \dAu \geq M_R > 0 \label{eq:817}.
\end{equation}
Note that $M_R$ is increasing in $R$, and that $\Dim( \phi_R V) = \Dim V$. Also note that since $S$ is compact, we may find $D_S$ such that $\int_\Sigma u^2 \dAu < D_S$ for all $u\in S$. Therefore, \eqref{eq:4_17} becomes
\begin{equation}
Q(\phi_R u, \phi_R u) \leq -\frac{M_R}{2} + \frac{2 D_S}{R^2},
\end{equation}
for all $u\in S$ and $R\geq R_0$. Taking $R\to\infty$, we may find $R$ independent of $u$ such that $Q(\phi_R u, \phi_R u) < 0$ for all $u\in S$. So we have that $\Dim V = \Dim(\phi_R V)$ and that $Q$ is negative definite on $\phi_R V$.

\end{proof}


Now, we use the space $\phi V$ from Lemma \ref{lem:phiV} and dimension counting to show that bounds on $\Ind Q$ forces $\Sigma$ to split off a linear space.

\begin{theorem} \label{th:index}

Consider any two-sided, smooth, properly immersed, non-planar hyper-surface $\Sigma\subset\mathbb R^{n+1}$ such that $\Au(\Sigma)<\infty$, $\Sigma$ satisfies the mean curvature condition $H = (1/2)\la x, N \ra + C$, and $\Ind Q \leq n$. Then there exists an $i$ such that $ n+1-\Ind Q \leq i \leq n$, and we have that
\begin{equation}
\Sigma = \Sigma_0 \times \mathbb R^{i}.
\end{equation}
Furthermore, for such non-planar $\Sigma$ it is impossible that $\Ind Q = 0$ or $\Ind Q=1$.
\end{theorem}

\begin{proof}
Let $V \equiv \Span\{1, \la v, N \ra\}_{v\in \mathbb R^{n+1}}$. First, we comment on $\Dim V$. Consider the case that the constant function $c_0 \in \Span\{\la v, N \ra\}_{v\in\mathbb R^{n+1}}$. We have that $L c_0 = (1/2) c_0$, but $c_0$ is also constant, so $L c_0 = (1/2+|A|^2) c_0$. Therefore, $|A|^2 c_0 \equiv 0$, and since $\Sigma$ is non-planar, we have that $c_0 = 0$. Hence,
\begin{equation}
\Dim V = 1 + \Dim\Span \{\la v, N \ra\}_{v\in \mathbb R^{n+1}}\label{eq:820}.
\end{equation}
By Lemma \ref{lem:phiV}, we have for some $\phi \in C^\infty_0(\Sigma)$ that $\Dim \phi V = \Dim V$ and that $Q$ is negative definite on $\phi V$. 

Remember that for $\Vu$ preserving variations we restrict to functions $\{u\in C^\infty_0 : \int_\Sigma u \dAu = 0\}$. So, we need to consider the space $\phi V \cap 1^\perp$. Now, note that by counting dimensions, we have $\Dim (\phi V \cap 1^\perp) \geq \Dim\Span \{\la v, N \ra\}_{v\in \mathbb R^{n+1}}.$ Hence,  $\Dim\Span \{\la v, N \ra\}_{v\in \mathbb R^{n+1}} \leq \Ind Q$. Considering the kernel of the linear transformation $\mathbb R^{n+1} \to C^\infty(\Sigma)$ given by $v \to \la v, N \ra$, we have that
\begin{align}
\Dim\{v: \la v, N \ra \equiv 0\} & = n+1 - \Dim\Span\{\la v, N \ra\}_{v\in \mathbb R^{n+1}}\\
 & \geq n+1 - \Ind Q.
\end{align}
 Finally, note that
\begin{equation}
\Sigma = \Sigma_0 \times \{v : \la v, N \ra \equiv 0\} \label{eq:822}.
\end{equation}
\end{proof}

\begin{remark}
Note that for the case of $\Sigma = S^n_R \subset \mathbb R^{n+1}$, one has that the eigenvalues of $\triangle_{ S_R^n}$ are given by $k(k+n-1)/R^2$ for $k=0,1,2,...$, and each eigenspace is given by the restriction of harmonic polynomials in $x_1,...,x_{n+1}$ that are homogeneous of degree $k$.

Therefore, for $\Sigma = S^n_R$, we have that $L$ has eigenvalues 
\begin{equation}
\lambda_k = \frac{1}{R^2}\left(k(k+n-1) - n\right) - \frac{1}{2}.
\end{equation}
The lowest eigenspace is given by the constant functions, so all other eigenspaces represent variations preserving $\Vu$.

The next eigenspace, for $\lambda_1 = -1/2$, is given by the functions $\{\la v, N\ra : v \in\mathbb R^{n+1}\}$. Note that its dimension is $n+1$.

The next eigenvalue is $\lambda_2 = (n+2)/R^2 - (1/2) $. So we see that for $R^2 < 2n + 4$, we have that $\Sigma = S^n_R$ is an example of a hyper-surface satisfying $H = (1/2)\la x, N \ra + C$ for some $C$, $\Ind Q = n+1$, and $\Sigma$ does not split off a linear space. Therefore, the index bound in Theorem \ref{th:index} is sharp.
\end{remark}

\begin{corollary}\label{co:gisop}
The hyper-planes are the only two-sided, smooth, complete, properly immersed hypersurfaces $\Sigma\subset\mathbb R^{n+1}$ such that $\Au(\Sigma)<\infty$, $\Sigma$ satisfies $H = (1/2)\la x, N\ra + C$ for some constant $C$, and $\Sigma$ satisfies the locally stable condition \eqref{eq:minimizing}. 

Furthermore, there are no two-sided, smooth, complete, properly immersed $\Sigma$ such that $\Au(\Sigma)<\infty$, $\Sigma$ satisfies $H = (1/2)\la x, N \ra + C$, and $\Ind Q = 1$.  
\end{corollary}


\section{An Integral Curvature Estimate} \label{sc:intest}

Using stability inequalities to obtain integral estimates for $|A|$ and then turning these estimates into pointwise estimates for $|A|$ is a long established technique in geometric analysis, see Schoen-Simon-Yau \cite{SchoenSimonYau}. In this Section and in Section \ref{sc:ptest}, we use this technique to get estimates on $|A|$. However, the one complication that we must deal with is that we can not put any test function $u\in C^\infty_0(\Sigma)$ into the stability condition \eqref{eq:minimizing}.

Since, the functions $\la v,N \ra$ for $v \in \mathbb R^{n+1}$ play a key role in the proof of theorem \ref{th:index}, it is not surprising that they play a key role in creating an integral estimate for the non-complete case. We will use these functions with appropriate cut-off functions to prove our estimate.

First, we need some notation. For two-sided $\Sigma$ and any $\phi \in C^\infty_0(\Sigma)$ such that $\phi \geq 0$ and $\phi \not\equiv 0$, let
\begin{equation}
N_\phi \equiv \frac{\intS \phi N \dAu}{\intS \phi\dAu} \label{eq:91}.
\end{equation}
In the case that $\phi\equiv 0$, we may define $N_\phi = 0$.

We find a modified version of the stability condition \eqref{eq:minimizing} that is valid for any $\phi \in C^\infty_0(\Sigma)$ such that $\phi\geq 0$. That is, we don't require that $\int_\Sigma \phi \dAu = 0$. This will allow us to use more standard cut-off functions $\phi \in C^\infty_0(\Sigma)$. One should compare \eqref{eq:92} to the stability inequality for minimal hyper-surfaces in Euclidean space: $\int_\Sigma \phi^2 |A|^2 \dA \leq \int_\Sigma |\nabla \phi|^2 \dA.$

\begin{lemma}\label{lm:modstability}
Let $\Sigma \subset \mathbb R^{n+1}$ be a two-sided, smooth, immersed hyper-surface satisfying the mean curvature condition $H = (1/2)\la x, N \ra + C$ and satisfying the stability condition \eqref{eq:minimizing}. Let $\phi \in C^\infty_0(\Sigma)$ such that $\phi \geq 0$ and $\phi \not \equiv 0$. Then, we have that
\begin{equation}
\intS \phi^2 |N - N_\phi|^2 \dAu + |N_\phi|^2 \intS \phi^2 |A|^2 \dAu \leq B_n \intS |\nabla \phi |^2 \dAu\label{eq:92}.
\end{equation}
Here $B_n$ is a constant depending on $n$.
\end{lemma}

\begin{proof}
Let $v \in \mathbb R^{n+1}$ such that $|v| = 1$, and note that $\int\nolimits_\Sigma \phi \la v, N - N_\phi \ra  \du = 0$. Therefore, we may plug $u = \phi \la v, N - N_\phi \ra$ into the stability condition \eqref{eq:minimizing}. A calculation similar to the proof of equation \eqref{eq:4_13} gives us that

\begin{equation}
\frac{1}{2} \intS \phi^2 \la v, N-N_\phi \ra^2 \dAu \leq \intS \phi^2 |A|^2 \la v, N_\phi \ra \la v, N - N_\phi \ra \dAu + \intS 4 |\nabla \phi|^2 \dAu\label{eq:94}.
\end{equation}

Applying a Cauchy inequality of the form $2ab \leq (1/2)a^2 + 2b^2$ to \eqref{eq:orthoA}, we get
\begin{equation}
\la v, N_\phi \ra \intS |A|^2 \phi^2 \la v, N \ra \dAu \leq \frac{1}{2} \intS \phi^2 \la v, N_\phi \ra ^2 |A|^2 \dAu + 2 \intS |\nabla \phi| ^2 |v^T|^2 \dAu\label{eq:95}.
\end{equation}

Combining \eqref{eq:94} and \eqref{eq:95} gives us
\begin{equation}
\intS \phi^2 \la v, N - N_\phi \ra ^2 \dAu +\la v, N_\phi\ra^2 \intS \phi^2 |A|^2 \dAu \leq 12 \intS |\nabla \phi|^2 \dAu \label{eq:96}.
\end{equation}

We sum over a constant orthonormal frame for $\mathbb R^{n+1}$ to prove the lemma.

\end{proof}


Now, we use this modified stability inequality to obtain an integral estimate for $|A|^2$.

\begin{theorem} \label{th:IntEst}
Let $\Sigma \subset B_{2R}(0) \subset \mathbb R^{n+1}$ with $\partial \Sigma \subset \partial B_{2R}(0)$ satisfy $H = (1/2)\la x, N \ra + C$ and the stability condition \eqref{eq:minimizing}. If  $\Au(\Sigma \cap B_R) \geq 2B_n R^{-2}\Au(\Sigma\cap(B_{2R}\setminus B_R))$, then we have that
\begin{equation}
\int\limits_{B_R\cap\Sigma} |A|^2 \dAu \leq 2B_n R^{-2} \Au(\Sigma \cap(B_{2R} \setminus B_R)). \label{eq:IntEst}
\end{equation}
Here, $B_n$ is the constant from Lemma \ref{lm:modstability}.
\end{theorem}

\begin{proof}
We first construct a cut off function depending only on the Euclidean $|x|$ such that
\begin{equation}
\phi(r) =
\begin{cases}
1 & r \leq R \\
\text{linear} & R \leq r \leq 2R \\
0 & 2R \leq r
\end{cases} \label{eq:98}
\end{equation}
Our modified stability inequality lemma gives us
\begin{equation}
\intS \phi^2 \dAu - 2|N_\phi|\intS \phi^2 \dAu + |N_\phi|^2\intS \phi^2(|A|^2 + 1) \dAu  \leq B_nR^{-2} \Au(\Sigma\cap(B_{2R}\setminus B_R)).\label{eq:99}
\end{equation}
Note that the left hand side of this inequality is quadratic in $|N_\phi|$, and since any quadratic with $a > 0$ satisfies $a u^2 + b u + c \geq c - \frac{b^2}{4a}$, we get that
\begin{equation}
\intS \phi^2 \dAu - \frac{\left(\intS \phi^2\dAu\right)^2}{\intS \phi^2 (|A|^2 + 1) \dAu} \leq B_n R^{-2}\Au(\Sigma\cap(B_{2R}\setminus B_R))\label{eq:910}.
\end{equation}
So, we have
\begin{equation}
\frac{\intS \phi^2 \dAu \intS \phi^2 |A|^2 \dAu}{\int\limits_\Sigma\phi^2 (1+|A|^2) \dAu} \leq B_n R^{-2}\Au(\Sigma\cap(B_{2R}\setminus B_R))\label{eq:911}.
\end{equation}
This inequality is of the form $\frac{ab}{a+b} \leq c$ where $a = \int_\Sigma \phi^2 \dAu$ and $b=\int_\Sigma \phi^2 |A|^2 \dAu$. This can be put into the form $(a - c)b \leq ca$. From our assumption that $\Au(\Sigma \cap B_R) \geq 2B_n R^{-2}\Au(\Sigma\cap(B_{2R}\setminus B_R))$ we get that $a \geq 2c$ and $a-c \geq (1/2)a$. Therefore, we have that $b \leq 2c$, which gives us that
\begin{equation}
\intS \phi^2 |A|^2 \dAu \leq 2B_nR^{-2}\Au(\Sigma\cap(B_{2R}\setminus B_R)).
\end{equation}

So, we get that
\begin{equation}
\int\limits_{B_R(0)} |A|^2 \dAu \leq 2B_n R^{-2}\Au(\Sigma\cap(B_{2R}\setminus B_R)).
\end{equation}

\end{proof}

\begin{remark}

For the case of properly immersed, two-sided, smooth, complete $\Sigma \subset\mathbb R^{n+1}$ satisfying the mean curvature condition \eqref{eq:curvcond}, stability condition \eqref{eq:minimizing}, and $\Au(\Sigma)< \infty$, we have that there exists an $R_0$ large enough such that for $R> R_0$, we have that $\Au(\Sigma \cap B_R) \geq 2B_n R^{-2}\Au(\Sigma\cap(B_{2R}\setminus B_R))$. Sending $R \to \infty$ in \eqref{eq:IntEst}, we get that $\int_\Sigma |A|^2 \dAu = 0$. So therefore, our estimate \eqref{eq:IntEst} also gives that the only such $\Sigma$ are hyper-planes.
\end{remark}

\begin{remark}
When considering $\Sigma$ non-complete, there are conditions that are sufficient to guarantee that $\Au(\Sigma \cap B_R) \geq 2B_n R^{-2}\Au(\Sigma\cap(B_{2R}\setminus B_R))$. Let $H = (1/2)\la x, N \ra + C$ with $|C| \leq M$.

We need a lower bound on $\Au(\Sigma\cap B_R)$. In order to accomplish this, we look at getting some control over $\min |x|$ and the Euclidean mean curvature $H$ around some point realizing $\min |x|$.  Let $\Sigma $ achieve $\min |x|$ at the point $p \in \Sigma$. At $p$, we have that
\begin{equation}
2n - |x|^2 + 2M |x| \geq 2n - |x|^2 - 2C\la x, N \ra = \mathcal L |x|^2 \geq 0.\label{eq:913}
\end{equation}

So, there exists a constant $D(M,n)$ large such that $\min |x| \leq D(M,n)$, and that $|H| \leq 2D(M,n)$ on $B_{2D(M,n)}$. Using Corollary \ref{co:vlower} and renaming $D(M,n)$, we can turn these bounds into a lower bound on the Euclidean area $\mathcal{A}(\Sigma\cap B_{2D(M,n)})\geq D(M,n)^{-1}$. Again renaming $D(M,n)$, we get that $\Au(\Sigma\cap B_{2D(M,N)})\geq e^{-D(M,n)^2} D(M,n)^{-1}$. Upon renaming constants, if $R > 2D(M,n)$, then $R^2 \Au(\Sigma \cap B_R) \geq R^2D(M,n)^{-1}$.

Therefore, there exists $D(M,n)$ such that if $\Au(\Sigma \cap B_{2R}) \leq D(M,n)^{-1} R^2$, then we are guaranteed that $\Au(\Sigma \cap B_R) \geq 2B_n R^{-2}\Au(\Sigma\cap(B_{2R}\setminus B_R))$.

\end{remark}

%
%

\section{A  Pointwise Curvature Estimate} \label{sc:ptest}

To achieve a pointwise curvature estimate from an integral estimate, we will need to make use of two inequalities: a Simons-type inequality and a Mean Value Inequality.  These inequalities have well-known analogues in the theory of minimal surfaces, and we adjust these proofs to fit our needs. The proof of the Mean Value Inequality, Lemma \ref{lm:mean}, is left to the Appendix.

A key element to these proofs are that we are in the case of $n=2$, that is $\Sigma \subset \mathbb R^3$. These proofs do not hold for general $n$.

\begin{lemma} \label{lm:simon}
\textbf{Simons Inequality:} For a  hyper-surface $\Sigma$ satisfying $H = (1/2)\la x, N \ra + C$, we have that
\begin{equation}
\triangle |A|^2 \geq -(|x|^2 / 8)|A|^2 - (2 + C^2)|A|^4 .\label{eq:12}
\end{equation}
\end{lemma}
\begin{proof}
First, note that from Codazzi's equation we have that $\nabla A$ is symmetric. We fix a point $p \in \Sigma$ and look at geodesic normal coordinates centered at $p$. Therefore, at $p$ we have that
\begin{equation}
\triangle A_{jk}  = \nabla^2_{jk}H + HA^2_{jk} - |A|^2 A_{jk} \label{eq:21}.
\end{equation}

Now, from $H = \frac{\la x, N \ra}{2} + C$ we have that
\begin{equation}
\nabla^2_{jk} H = (1/2)\nabla_j A(k, x^T) + (1/2)A_{jk} - (1/2)\la x, N \ra A^2_{jk} \label{eq:23}.
\end{equation}

Applying a Cauchy-Schwarz inequality of the form $ab \leq 2 a^2 + (1/8)b^2$ to $\la A, \nabla_{x^T} A\ra$, we get
\begin{equation}
\triangle |A|^2 \geq -(|x|^2 / 8) |A|^2 - (2+C^2)|A|^4 \label{eq:25}.
\end{equation}
\end{proof}

\begin{lemma} \label{lm:mean}
\textbf{Mean Value Inequality:} Suppose that, on a hypersurface with $\abs{H} \leq M$, a function $f$ satisfies $f\geq 0$ and $\Delta f \geq -\lambda t^{-2}f$ for some $\lambda$ on $B_t (x)$.  Then, for $s\leq t$ we have that
\begin{equation}
e^{\left(\lambda/2t + M\right)s}s^{-n}\int\limits_{B_s (x) \cap \Sigma} f \dA\geq \omega_n f(x).
\label{eq:13}
\end{equation}
\end{lemma}

We now use a standard tool due to Choi-Schoen \cite{choischoen} for turning our integral estimates for $|A|$ into pointwise estimates. Note, that in Theorem \ref{th:choischoen}, we need to require that $r_0 < 1/R$. This is to give more control of estimates coming from the Mean Value Inequality \eqref{eq:13} to give us \eqref{eq:211}. Also it is used to control the $|x|^2$ term in our Simons-type inequality \eqref{eq:12}.

\begin{theorem}\label{th:choischoen}
There exists $\epsilon_M > 0$ such that the following holds: Suppose $\Sigma\subset \mathbb R^3$ is any hypersurface satisfiying $H = \la x, N \ra /2 + C$ with $|C| \leq M$, and suppose $x_0 \in \Sigma$.  Also, suppose that for some $R \geq 1$ and some $r_0 < 1/R$, we have $B_{r_0}(x_0) \subset B_R(0)$ and $\partial \Sigma \subset \partial B_R(0)$.  Finally, suppose that
\begin{equation}
\int\limits_{B_{r_0}(x_0)} |A|^2\dA < \delta \epsilon \label{eq:ChoiSchoen}.
\end{equation}
Then for all $0 < \sigma \leq r_0$ and $y\in B_{r_0 - \sigma}(x_0)$, $|A|^2(y) \leq \delta/\sigma^2$.
\end{theorem}

\begin{proof}
On $B_{r_0}(x_0)$, define the function
\begin{equation}
F(y) = \left( r_0 - d(y,x_0)\right)^2 \abs{A(y)}^2
\label{eq:14}
\end{equation}

\noindent where $d(y, x_0)$ is the Euclidean distance between the two points.  Observe that $F\geq 0$ in $B_{r_0}$, and $F=0$ on $\partial B_{r_0}$.  Set $x_1$ to be the point where $F$ achieves its maximum.  Observe that if $F(x_1) \leq \delta$ we will be done, since for $y \in B_{r_0 - \sigma}$, $\sigma^2 \abs{A}^2 \leq F(x_1) \leq \delta$.  We will now show that $F(x_1) > \delta$ gives a contradiction for some $\epsilon_M$ small enough and independent of $\Sigma, \delta,$ and $R\geq 1$.

 Suppose that $F(x_1) > \delta$, ie
\begin{equation}
\left(r_0 - d(x_1 , x_0)\right)^2 \abs{A(x_1)}^2 > \delta,
\label{eq:15}
\end{equation}
and fix $\sigma$ so that $\sigma^2 \abs{A}^2 = \delta/4$.  Observe that the following equations hold:
\begin{align}
\sigma &\leq \frac{1}{2}\left(r_0 - d(x_1 , x_0)\right) \leq \frac{1}{2R} < 1 \label{eq:16},\\
\frac{1}{2} &\leq \frac{r_0 - d(y,x_0)}{r_0 - d(x_1 , x_0)} \leq 2, \forall y \in B_{\sigma}(x_1) \label{eq:17}.
\end{align}

Using these, we compute
\begin{align}
(r_0 - d(x_1 , x_0))^2 \sup\limits_{B_\sigma (x_1)} \abs{A}^2 &\leq 4 \sup\limits_{B_\sigma (x_1)} (r_0 - d(\cdot , x_0))^2 \abs{A}^2 ,\notag\\
&= 4 \sup\limits_{B_\sigma (x_1)} F(\cdot) \leq 4F(x_1), \notag\\
&= 4(r_0 - d(x_1 , x_0))^2\abs{A}^2(x_1)\notag.
\end{align}
Therefore,
\begin{equation}
\sup\limits_{B_\sigma (x_1)} \abs{A}^2 \leq 4 \abs{A}^2 (x_1) = \frac{\delta}{\sigma ^2} \label{eq:18} < \frac{1}{\sigma^2}.
\end{equation}

Plugging \eqref{eq:18} into Simons Inequality \eqref{eq:12}  gives us

\begin{equation}
\Delta \abs*{A}^2 \geq - (R^2/8)|A|^2 - (2 + C^2)/\sigma^2 |A|^2,
\end{equation}

\noindent and using \eqref{eq:16} (specifically, that $R \leq 1/\sigma$) yields

\begin{equation}
\Delta \abs*{A}^2 \geq - \sigma^{-2} (3 + M^2) |A|^2.
\end{equation}

\noindent Therefore,

\begin{equation}
\Delta \abs*{A}^2 \geq -\lambda\sigma^{-2}\abs{A}^2  \label{eq:110}
\end{equation}

\noindent on $B_\sigma (x_1)$, where $\lambda = \lambda(M) =  3+M^2$.  Note that $|H| \leq R + M$. Then, by the Mean Value Inequality \eqref{eq:13},

\begin{align}
\abs{A}^2(x_1) & \leq\omega_n^{-1} e^{R\sigma + M\sigma + (3+M^2)/2} \sigma^{-2}\int\limits_{B_\sigma (x_1) \cap \Sigma} \abs{A}^2 \dA,\\
& \leq \omega_n^{-1} e^{M + 2 + M^2/2} \sigma^{-2}  \int\limits_{B_\sigma (x_1) \cap \Sigma} \abs{A}^2\dA \label{eq:211}.
\end{align}
Here for \eqref{eq:211}, we have used that $\sigma R \leq 1$ which is a consequence of our hypothesis that $r_0 < 1/R$.

%
%

\noindent Substituting back in our definition of $\sigma$, we get

\begin{align}
\delta/4 = \sigma^{2}\abs{A}^2(x_1) &\leq \omega_n^{-1} e^{M + 2 + M^2/2} \int\limits_{B_\sigma (x_1) \cap \Sigma} \abs{A}^2\dA,\\
 &\leq \omega_n^{-1} e^{M + 2 + M^2/2} \int\limits_{B_{r_0} (x_1) \cap \Sigma} \abs{A}^2 \dA,\\
 &\leq \omega_n^{-1} e^{M + 2 + M^2/2} \delta \epsilon.
\label{eq:212}
\end{align}

We see that we may choose $\epsilon$ depending only on $M$ such that there is a contradiction. 

\end{proof}


Using Theorem \ref{eq:ChoiSchoen} combined with Theorem \ref{eq:IntEst} we get pointwise estimates for $\Vu$ preserving stable hyper-surfaces.

\begin{theorem}[Pointwise for $n=2$] \label{th:Pointwise}
Let $M > 0$ be given and $R> 1$. Also, let $\Sigma \subset B_{2R}(0) \subset \mathbb R^{3}$ with $\partial \Sigma \subset \partial B_{2R}(0)$ be a hyper-surface with $H = \la x, N \ra/2 + C$ and $|C| \leq M$ that satisfies the stability condition \eqref{eq:minimizing}.  

There exists $\epsilon_M > 0$ such that if $\Au(\Sigma \cap B_R) \geq 2B_n R^{-2}\Au(\Sigma\cap(B_{2R}\setminus B_R))$ and $\Au(\Sigma \cap(B_{2R} \setminus B_R)) < (R^2/2B_n)e^{-(\frac{1}{16}+\gamma) R^2} \epsilon_M $ for some $\gamma > 0$, then

\begin{equation}
\sup\limits_{x\in B_{R/4}(0)} |A|^2 \leq 16 R^2e^{-\gamma R^2}.
\end{equation}
\end{theorem}
\begin{remark}
Note that in the case that $\mathcal{A}(\Sigma) \leq D R^n$, one does indeed get that for large enough $R$ that $\Au(\Sigma \cap(B_{2R} \setminus B_R)) < (R^2/2B_n)e^{-(\frac{1}{16}+\gamma) R^2} \epsilon_M $ for some $\gamma > 0$.
\end{remark}

%
%
\appendix

%
%

\section{Mean Value Inequality} \label{sc:mean}

Here we give a proof of the Mean Value Inequality, Lemma \ref{lm:mean}. The techniques are well-known (see Colding-Minicozzi \cite{cmbook}), but we include a proof for completeness.

\begin{proof}[Proof of Lemma \ref{lm:mean}]
Assume $\abs*{H} \leq M$. Lemma \ref{lm:mean} is stated in terms of Euclidean quantities, so we are free to translate so that we are considering $B_s(0)$. Recall that $\Delta \abs*{x}^2 = 2n - 2\la x,N \ra H$. Then
\begin{align}
2n\int\limits_{B_s \cap \Sigma} f \dA &= \int\limits_{B_s \cap \Sigma} f\Delta \abs*{x}^2 \dA+ 2\int\limits_{B_s \cap \Sigma}f\la x, N\ra H \dA, \label{eq:31}\\
&= \int\limits_{B_s \cap \Sigma}\abs*{x}^2 \Delta f \dA + 2\int\limits_{\partial B_s \cap \Sigma} f \abs*{x^T}\dA - s^2 \int\limits_{B_s \cap \Sigma} \Delta f \dA+ 2 \int\limits_{B_s \cap \Sigma} \la x,N \ra Hf \dA\notag.
\end{align}

\noindent Let $g(s) = s^{-n}\int\limits_{B_s \cap \Sigma} f\dA$. Using the coarea formula and \eqref{eq:31}, we get

\begin{align}
g'(s) \geq \frac{1}{2}s^{-n+1}\int\limits_{B_s \cap \Sigma} \Delta f \dA- s^{-n-1}\int\limits_{B_s \cap \Sigma}\la x, N\ra f H \dA
\label{eq:32}.
\end{align}
Here we have used the positivity of $f$.  Additionally, if we assume $\Delta f \geq -\lambda t^{-2} f$ on $B_t$, our bound on $\abs{H}$ gives us

\begin{align}
g'(s) &\geq \frac{-\lambda}{2} s^{1-n}\int\limits_{B_s \cap \Sigma}f t^{-2}\dA - Ms^{-1-n}\int\limits_{B_s \cap \Sigma}s f \dA,\notag\\
&\geq -\left(\frac{\lambda}{2t} + M\right)g(s)
\label{eq:33}
\end{align}

\noindent for all $s\leq t$.  Therefore,

\begin{equation}
\frac{d}{ds}\left( g(s)e^{\left(\frac{\lambda}{2t} + M\right)s}\right) \geq 0. \label{eq:34}
\end{equation}

Integrating \eqref{eq:34} from $s_0$ to $s_1$ (both assumed to be less than $t$) and letting $s_0 \searrow 0$, we get

\begin{equation}
e^{\left(\frac{\lambda}{2t} + M\right) s_1} s_1^{-n}\int\limits_{B_{s_1}\cap \Sigma} f \dA\geq \omega_nf(p). \label{eq:35}
\end{equation}

\end{proof}

Note, that we get the following corollary (monotonicity):

\begin{corollary} \label{co:vlower}
Let $p \in \Sigma$, and let $\abs{H} \leq M$ in $B_{t}(p) \cap \Sigma$.  Then for $s\leq t$, we have $\abs*{B_s \cap \Sigma} \geq \omega_n e^{-Ms}s^n$, where $\omega_n$ is the volume of the standard unit ball in $\mathbb{R}^n$.
\end{corollary}

\begin{proof}
Use the Mean Value Inequality, Lemma \ref{lm:mean} with $f\equiv 1$ and $\lambda = 0$.

\end{proof}


\bibliographystyle{plain}
\bibliography{VolumePreservingRef}

\end{document}